\documentclass[a4paper,titlepage,12pt]{article}
\usepackage{hyperref}
\usepackage[usenames,dvipsnames]{pstricks} 
\usepackage{pst-plot} 
\usepackage{amssymb,amsthm,amsmath} 
\usepackage[a4paper]{geometry}
\usepackage{datetime2}
\usepackage[utf8]{inputenc}
\usepackage[italian,english]{babel}
\usepackage{mathrsfs}
\usepackage{bbm}
\usepackage{enumitem}

\newcommand{\R}{\mathbb{R}}
\newcommand{\Z}{\mathbb{Z}}
\newcommand{\s}{\mathbb{S}}

\newcommand{\D}{\mathcal{D}}

\newcommand{\re}{\mathbb{R}}
\newcommand{\N}{\mathbb{N}}
\newcommand{\ep}{\varepsilon}

\newcommand{\Fgpl}{F_{\gamma,p,\lambda}}
\newcommand{\Egpl}{E_{\gamma,p,\lambda}}
\newcommand{\psigpl}{\psi_{\gamma,p,\lambda}}
\newcommand{\DFgpl}{DF_{\gamma,p,\lambda}}
\newcommand{\ggpl}{g_{\gamma,p,\lambda}}

\newcommand{\Vis}{\mathcal{V}}
\newcommand{\Imp}{\mathcal{I}}

\newcommand{\loc}{\mathrm{loc}}
\newcommand{\Leb}{\ensuremath{\mathscr{L}}}

\newtheorem{thm}{Theorem}[section]
\newtheorem{thmbibl}{Theorem}

\newtheorem{rmk}[thm]{Remark}
\newtheorem{prop}[thm]{Proposition}

\newtheorem{question}[thmbibl]{Question}

\title{Gamma-liminf estimate for a class of non-local approximations of Sobolev and BV norms}

\author{
Massimo Gobbino\vspace{1ex}\\ 
{\normalsize Università di Pisa} \\
{\normalsize Dipartimento di Matematica}\\ 
{\normalsize PISA (Italy)}\\  
{\normalsize e-mail: \texttt{massimo.gobbino@unipi.it}}
\and
Nicola Picenni\vspace{1ex}\\ 
{\normalsize Università di Pisa} \\
{\normalsize Dipartimento di Matematica}\\ 
{\normalsize PISA (Italy)}\\  
{\normalsize e-mail: \texttt{nicola.picenni@dm.unipi.it}}
}
\date{}

\begin{document}

\maketitle

\begin{abstract}

We consider a family of non-local and non-convex functionals, and we prove that their Gamma-liminf is bounded from below by a positive multiple of the Sobolev norm or the total variation. As a by-product, we answer some open questions concerning  the limiting behavior of these functionals.

The proof relies on the analysis of a discretized version of these functionals.

\vspace{6ex}

\noindent{\bf Mathematics Subject Classification 2020 (MSC2020):} 49J45, 26D10, 46E35, 26B30.

\vspace{6ex}

%49J45 Methods involving semicontinuity and convergence; relaxation
%26D10: Inequalities involving derivatives and differential and integral operators
%46E35 Sobolev spaces and other spaces of “smooth” functions, embedding theorems, trace theorems
%26B30 Absolutely continuous real functions of several variables, functions of bounded variation
%35B27 Homogenization in context of PDEs; PDEs in media with periodic structure
%35A15 Variational methods applied to PDEs
%35B25 Singular perturbations in context of PDEs
%35J20 Variational methods for second-order elliptic equations
%35J35 Variational methods for higher-order elliptic equations
%47J30 Variational methods involving nonlinear operators
%49Q20 Variational problems in a geometric measure-theoretic setting
%34E10 Perturbations, asymptotics of solutions to ordinary differential equations
%35B44 Blow-up in context of PDEs
%47J06 Nonlinear ill-posed problems
%35A35 Theoretical approximation in context of PDEs
%41A30 Approximation by other special function classes


\end{abstract}

\section{Introduction}

In this paper we answer some questions concerning a class of non-local, non-convex functionals introduced in~\cite{2022-BSVY-Lincei,BSVY}, and related to the Sobolev norm and the total variation. 

In order to introduce these functionals, we consider a positive integer $N$, an open set $\Omega \subseteq \R^N$, and three real parameters $\gamma\in\R$, $p\geq 1$ and $\lambda>0$. Following~\cite{BSVY}, we consider the measure $\nu_\gamma$ on $\R^N\times \R^N$ defined by
$$\nu_\gamma (A):=\iint_A |y-x|^{\gamma-N}\,dx\,dy,$$
and for every measurable function $u:\Omega\to\R$ we introduce the set
\begin{equation*}
\Egpl(u,\Omega):=\left\{(x,y)\in \Omega\times\Omega:|u(y)-u(x)|>\lambda |y-x|^{1+\gamma/p}\right\}.
\end{equation*}

We are now ready to define the family of functionals 
\begin{equation}\label{defn:Fgpl}
\Fgpl(u,\Omega):=\lambda^p \nu_\gamma(\Egpl(u,\Omega))
\qquad
\forall u\in L^1_\loc(\Omega).
\end{equation}

In order to describe some relations between the functionals (\ref{defn:Fgpl}), the Sobolev norms, and the total variation, let $\dot{W}^{1,p}(\Omega)$ denote the space of functions $u\in L^1 _{\loc}(\Omega)$ whose distributional gradient $\nabla u$ belongs to $L^p(\Omega)$, and let $\dot{BV}(\Omega)$ denote the space of functions $u\in L^1 _{\loc}(\Omega)$ whose distributional gradient $Du$ is a finite $\R^N$-valued Radon measure on $\Omega$. Then for every $u\in L^1 _{\loc}(\Omega)$ we set
\begin{equation}
F_p(u,\Omega):=\begin{cases}
\|\nabla u\|_{L^p(\Omega)} ^p & \mbox{if $p>1$ and $u\in \dot{W}^{1,p}(\Omega)$},
\\[0.5ex]
|Du|(\Omega) & \mbox{if $p=1$ and $u\in \dot{BV}(\Omega)$},
\\
+\infty & \mbox{otherwise}.
\end{cases}
\label{defn:Fp}
\end{equation}

We consider also the geometric constant 
\begin{equation}\label{defn:C_N}
C_{N,p}:=\int_{\s^{N-1}} |\langle v,x\rangle|^p\,d\mathcal{H}^{N-1}(x),
\end{equation}
where $v$ is any element of the unit sphere $\s^{N-1}$ with center in the origin of $\R^N$ (of course the definition does not depend on the choice of $v$).

\paragraph{\textmd{\textit{Previous literature and Brezis' open problems}}}

The functionals (\ref{defn:Fgpl}) generalize some families of functionals considered previously in the literature, and corresponding to special values of $\gamma$. 

In the case $\gamma=-p$, the family $\{F_{-p,\lambda,p}\}$ was first studied in \cite{2006-Nguyen-JFA,2011-Nguyen-Duke} (see also the more recent developments in \cite{2018-CRAS-AGMP,2020-APDE-AGMP,2020-APDE-AGP,2018-AnnPDE-BN,2020-CCM-BN}). In this special case the two main results are the following.
\begin{itemize}
    \item (Pointwise convergence) For every $p>1$ it turns out that
    $$\lim_{\lambda\to 0^+} F_{-p,p,\lambda}(u,\R^N)=
    \frac{C_{N,p}}{p}\cdot F_p(u,\R^N) 
    \qquad 
    \forall u \in L^p(\R^N).$$

    In the case $p=1$, the analogous limit holds true when $u\in C^1_c(\R^N)$, but not necessarily when $u\in W^{1,1}(\R^N)$.

    \item  (Gamma convergence)  For every $p\geq 1$ it turns out that
    $$\Gamma\mbox{--}\lim_{\lambda\to 0^+} F_{-p,p,\lambda}(u)=
    \frac{C_{N,p}}{p}\cdot k_{p}\cdot F_p(u,\R^N) 
    \qquad \forall u \in L^p(\R^N),$$
    where
    \begin{equation*}
    k_p:=\begin{cases}
    \dfrac{1}{p-1}\left(1-\dfrac{1}{2^{p-1}}\right) &\mbox{if $p>1$},
    \\[1.5ex]
    \log 2 &\mbox{if $p=1$}.
    \end{cases}
\end{equation*}
    Note that, quite surprisingly, the Gamma-limit is different from the pointwise limit.

\end{itemize}

The more recent literature \cite{2021-BVY-PNAS,2021-BVY-CalcVar,2022-Poliakovsky-JFA,Brezis-Lincei} (see also \cite{2022-JFA-DLYYZ,2023-CalcVar-DLYYZ}) focused on the special case $\gamma=N$, in which the measure $\nu_N$ is equal to the Lebesgue measure, and the supremum of $F_{N,p,\lambda}(u,\Omega)$ over all $\lambda>0$ is by definition the weak $L^p(\Omega\times\Omega)$ quasi-norm of the function
$$\frac{|u(y)-u(x)|}{|y-x|^{N+p}}.$$

Finally, in \cite{2022-BSVY-Lincei,BSVY} (see also \cite{2023-CalcVar-ZYY,2023-CCM-ZYY}) the general case $\gamma \in \R$ was considered. Here, motivated mainly by \cite{Brezis-Lincei}, we restrict to the case $\gamma>0$, which of course includes the case $\gamma=N$. Let us describe some of the results that were known in this setting (we stress that what is interesting in the case $\gamma<0$ is the limiting behavior as $\lambda\to 0^+$, while for $\gamma>0$ it is the limiting behavior as $\lambda\to+\infty$).

\begin{itemize}

\item (Global pointwise estimate from above and below) In~\cite[Theorem~1.3 and Theorem~1.4]{BSVY} it was proved that there exist two constants $c_1(N,\gamma,p)$ and $c_2(N,\gamma,p)$ such that
\begin{equation*}%\label{est:sup_F}
c_1(N,\gamma,p)F_p(u,\R^N)\leq \sup_{\lambda>0} \Fgpl(u,\R^N)\leq c_2(N,\gamma,p)F_p(u,\R^N),
\end{equation*}
for every $u\in L^1 _{\loc}(\R^N)$.

\item  (Pointwise convergence in Sobolev spaces) In~\cite[Theorem~1.1]{BSVY} it was proved that for every $p\geq 1$ it holds that
\begin{equation}\label{lim_W1p}
\lim_{\lambda\to +\infty} \Fgpl(u,\R^N)=
C_{N,p}\cdot\frac{1}{\gamma}\cdot F_p(u,\R^N) 
\qquad 
\forall u\in \dot{W}^{1,p}(\R^N).
\end{equation}

\item  (Pointwise convergence and estimates for BV functions) In the case $p=1$, the pointwise convergence result was extended and improved in several directions. In~\cite[Lemma~3.6]{BSVY} it was proved that, if $E\subseteq \R^N$ is a bounded convex set with smooth boundary, then
\begin{equation*}
\lim_{\lambda\to +\infty} F_{\gamma,1,\lambda}(\mathbbm{1}_E,\R^N)=
C_{N,1}\cdot\frac{1}{\gamma+1} F_1 (\mathbbm{1}_E,\R^N),
\end{equation*}
where $\mathbbm{1}_E$ is the function that is equal to~1 if $x\in E$ and 0 otherwise.

More recently, in~\cite[Theorem~1.1]{2023-Picenni}) it was proved that for every $u \in \dot{BV}(\R^N)$ it turns out that
\begin{multline*}\label{liminf_BV}
\liminf_{\lambda\to +\infty} F_{\gamma,1,\lambda}(u,\R^N) \geq
\\
C_{N,1}\left\{\frac{1}{\gamma} |D^a u|(\R^N) + \frac{1}{\gamma+1} |D^j u|(\R^N) + \frac{\gamma}{2(2\gamma+1)(\gamma+1)}|D^c u|(\R^N)\right\}
\end{multline*}
where $D^a u$, $D^j u$ and $D^c u$ denote, respectively, the absolutely continuous part, the jump part and the Cantor part of the distributional derivative of $u$. In addition, in ~\cite[Theorem~1.2]{2023-Picenni}) it was proved that
\begin{equation}\label{lim_SBV}
\lim_{\lambda\to +\infty} F_{\gamma,1,\lambda}(u,\R^N) = 
C_{N,1}\left\{\frac{1}{\gamma} |D^a u|(\R^N) + \frac{1}{\gamma+1} |D^j u|(\R^N)\right\}
\end{equation}
for every $u \in \dot{BV}(\R^N)$ with $D^c u=0$.

Finally, in~\cite[Theorem~1.1]{2023-Lahti}) the pointwise estimate from below in~\cite{2023-Picenni} was improved by showing that the constants in front of the jump part and the Cantor part are equal, namely
\begin{equation*}\label{liminf_BV_2}
\liminf_{\lambda\to +\infty} F_{\gamma,1,\lambda}(u,\R^N) \geq C_{N,1}\left\{\frac{1}{\gamma} |D^a u|(\R^N) + \frac{1}{\gamma+1} \left(|D^j u|(\R^N) + |D^c u|(\R^N)\right)\right\},
\end{equation*}
but again just for $u\in\dot{BV}(\R^N)$.

\end{itemize}

The extension of the convergence results and estimates quoted above to every $u\in L^1_{\loc} (\R^N)$ does not seem to be straightforward. Indeed, the following questions were raised in~\cite[Section~9]{Brezis-Lincei} and in~\cite[Section~7.2]{BSVY}.
\begin{question}\label{question_constant}
Let $\gamma>0$ and $p\geq 1$ be real numbers, and let $u\in L^1_{\loc}(\R^N)$ be a function such that
$$\liminf_{\lambda\to+\infty} \Fgpl(u,\R^N)=0.$$
Can we conclude that $u$ is constant (almost everywhere)?
\end{question}

\begin{question}\label{question_liminf}
Let $\gamma>0$ and $p\geq 1$ be real numbers. Is there a positive constant $c(N,\gamma,p)>0$ such that
$$\liminf_{\lambda\to+\infty} \Fgpl(u,\R^N)\geq c(N,\gamma,p) F_p(u,\R^N),$$ 
for every function $u\in L^1_{\loc}(\R^N)$?
\end{question}

\begin{question}\label{question_Gamma-lim}
Let $\gamma>0$ and $p\geq 1$ be real numbers. Is there a positive constant $c(N,\gamma,p)>0$ such that
$$\Gamma\mbox{--}\lim_{\lambda\to+\infty} \Fgpl(u,\R^N) = c(N,\gamma,p) F_p(u,\R^N),$$ 
for every function $u\in L^1_{\loc}(\R^N)$?
\end{question}

We conclude by mentioning that in~\cite{2022-Poliakovsky-JFA} a positive answer is given, in the case $\gamma=N$, to the weaker version of the first two questions with the limsup instead of the liminf. The proof relies on the so-called BBM formula (see~\cite{BBM,Davila-BBM}).

\paragraph{\textmd{\textit{Our contribution}}}

In this paper we give a complete answer to Question~\ref{question_constant} and Question~\ref{question_liminf} as stated above, and a partial answer to Question~\ref{question_Gamma-lim}. In fact, our main result is the following.

\begin{thm}\label{thm:main}
Let $\gamma>0$ and $p\geq 1$ be real numbers, let $N$ be a positive integer, and let $\Omega\subseteq \R^N$ be any open set. 

Then the families of functions defined by (\ref{defn:Fgpl}) and (\ref{defn:Fp}) satisfy
\begin{equation}
\Gamma\mbox{--}\liminf_{\lambda\to+\infty} \Fgpl(u,\Omega) \geq 
C_{N,p}\cdot c_\gamma\cdot F_p(u,\Omega) 
\qquad 
\forall u \in L^1 _{\loc} (\Omega),
\label{th:main}
\end{equation}
where $C_{N,p}$ is the geometric constant defined by (\ref{defn:C_N}), and
\begin{equation}\label{defn:c_gamma}
c_\gamma:=\frac{\log 2}{2^{\gamma+1}-1}.
\end{equation}

More precisely, for every $u\in L^1 _{\loc}(\Omega)$ and every family of functions $\{u_\lambda\}\subseteq L^1_{\loc}(\Omega)$ such that $u_\lambda \to u$ in $L^1 _{\loc}(\Omega)$ as $\lambda \to +\infty$, it holds that
\begin{equation*}
\liminf_{\lambda\to+\infty} \Fgpl(u_\lambda,\Omega) \geq 
C_{N,p}\cdot c_\gamma\cdot F_p(u,\Omega).
\end{equation*}
\end{thm}

Let us comment a little on the state of the art after our result.

\begin{rmk}[Toward the pointwise limit]
\begin{em}

When $p>1$ now we know that the pointwise limit is given by (\ref{lim_W1p}) for every $u\in L^1_\loc(\R^N)$, and this settles the issue of pointwise convergence for $p>1$.

When $p=1$ now we know that the pointwise limit is given by (\ref{lim_SBV}) both if $u\in \dot{BV}(\R^N)$ with $D^c u=0$, and if $u\in L^1_\loc(\R^N)\setminus\dot{BV}(\R^N)$. Moreover, the estimate from below in~\cite{2023-Lahti} now holds true for every $u\in L^1_\loc(\R^N)$.

It should be possible to extend these results to open sets $\Omega\subseteq\R^N$ with smooth enough boundary, or to any open set if one restricts the set $\Egpl(u,\Omega)$ to the pairs $(x,y)$ that ``see each other'' (see~\cite[Section~7]{1998-CPAM-Gobbino} and~\cite[Remark~4.6]{2001-PRSE-GM}).

What remains open is the case where $p=1$ and $u$ is a bounded variation function whose distributional derivative has a nontrivial Cantor part. In this case, as far as we know, even the existence of the pointwise limit is unclear.

\end{em}
\end{rmk}

\begin{rmk}[Toward the Gamma-limit]
\begin{em}

In the case $p>1$ the pointwise convergence result (\ref{lim_W1p}) implies that
\begin{equation}
\Gamma\mbox{--}\limsup_{\lambda\to +\infty} \Fgpl(u,\Omega) \leq 
\frac{C_{N,p}}{\gamma} \cdot F_p(u,\Omega)
\qquad 
\forall u \in L^1 _{\loc} (\Omega),
\label{eqn:Gamma-p>1}
\end{equation}
at least when $\Omega$ is smooth enough, while in the case $p=1$ the pointwise convergence result (\ref{lim_SBV}) implies that
\begin{equation}
\Gamma\mbox{--}\limsup_{\lambda\to +\infty} F_{\gamma,1,\lambda}(u,\Omega) \leq 
\frac{C_{N,1}}{\gamma+1} \cdot F_1(u,\Omega)
\qquad 
\forall u \in L^1 _{\loc} (\Omega),
\label{eqn:Gamma-p=1}
\end{equation}
again when $\Omega$ is smooth enough.

For the time being, the existence of the Gamma-limit remains unclear. It might be possible to extend the approach of~\cite{2011-Nguyen-Duke} in order to provide a positive answer to Question~\ref{question_Gamma-lim} with some constant $C(N,\gamma,p)$ (positive because of Theorem~\ref{thm:main}) defined as the solution of an appropriate cell problem. On the other hand, this approach would not answer the question whether the Gamma-limit coincides with the (relaxation of) the pointwise limit, namely with the right-hand side of (\ref{eqn:Gamma-p>1}) for $p>1$, and of (\ref{eqn:Gamma-p=1}) if $p=1$.

Now the challenge is to improve the estimate from below in (\ref{th:main}). Indeed, we suspect that the constant $c_\gamma$ given by (\ref{defn:c_gamma}) is non-optimal, because in the proof we combine several inequalities that are optimal when considered one by one, but with different equality cases. Alternatively, it would be interesting to find an example of a family $u_\lambda\to u$ for which the limit of $\Fgpl(u_\lambda,\Omega)$ is strictly less than the right-hand side of (\ref{eqn:Gamma-p>1}) or (\ref{eqn:Gamma-p=1}), as it happens in the case $\gamma=-p$.

\end{em}
\end{rmk}

\paragraph{\textmd{\textit{Overview of the technique}}}

The first steps of the proof follow a rather standard path.

First of all, we reduce ourselves to the one-dimensional case through an integral geometric approach. The main idea is that the functionals (\ref{defn:Fgpl}) and (\ref{defn:Fp}) are in some sense the average of the same functionals computed on the one-dimensional slices of $u$, namely the restrictions of $u$ to all possible lines in $\R^N$. We refer to section~\ref{sec:proof-HD} for the details.

We stress that here the flexibility provided by the parameter $\gamma$ plays a role. Indeed, the problem in dimension $N$ with $\gamma=N$ reduces to the one-dimensional problem with $\gamma=N$, and therefore solving the one-dimensional case with $\gamma=1$ is not enough to answer the questions posed in~\cite{Brezis-Lincei} in the case $\gamma=N$.

After reducing to dimension one, we localize the problem, and we reduce ourselves to estimating from below the cost of an oscillation. Roughly speaking, this amounts to considering a sort of cell problem of the following form.

\begin{quote}
    Suppose that $u:\R\to\R$ is a measurable function such that $u(x)=A$ for $x\leq a$, and $u(x)=B$ for $x\geq b$. Can we estimate $\Fgpl(u,\R)$ from below in terms of $|B-A|$ and $b-a$?
\end{quote}

We answer this question in a quantitative way in Proposition~\ref{prop:gatto}. In order to prove this result, we generalize the approach of~\cite{1998-CPAM-Gobbino,2001-PRSE-GM}. The original idea was to fix $x\in\R$ and a step $\delta>0$, and then considering a discretized version of the functional that takes into account only the values of $u$ in the points of  the form $x+\delta i$, with $i\in\Z$. Here this is not enough, and our discretized version of (\ref{defn:Fgpl}) has to take into account all the values of $u$ in the points of the form $x+\delta i/2^k$, with $i\in\Z$ and $k\in\N$. This lead us to introduce in (\ref{defn:DF}) what we call the dyadic functional.

Again, the main idea is that the functional (\ref{defn:Fgpl}) in dimension one is the integral of the dyadic functional over all possible choices of $x$ and $\delta$ (see Proposition~\ref{prop:dyadic_repr}), and hence we have reduced ourselves to estimating from below the dyadic functional. Since we are now in an almost discrete setting, this is a sort of combinatorial problem. For the details we refer to Proposition~\ref{prop:est_DF}, that is the technical core of the paper.

\paragraph{\textmd{\textit{Comparison with similar discretization arguments for non-local functionals}}}

The idea of reducing the computation of Gamma-limits to cell problems is classical, and goes back to the dawning of Gamma-convergence. In the context of non-local functionals, the estimate for the cell problem is usually obtained through some kind of discretization of the energy. 

In~\cite{2006-CRAS-BN} (see also \cite{2008-JEMS-Nguyen,2011-Nguyen-Duke,2011-CalcVar-Nguyen}) the functionals (\ref{defn:Fgpl}) with $\gamma=-p$ are considered, and the discretization is performed in a ``vertical'' way (or \emph{à la Lebesgue}, as observed by J.~M.~Morel), namely by considering the sets $A_j:=\{x:(j-1)2^{-k}\leq u(x)<j2^{-k}\}$, where $k$ is a suitably chosen but \emph{fixed} positive integer, and by estimating the contributions to the integral of the pairs $(x,y)\in A_{j_1} \times A_{j_2}$ with $j_1\neq j_2$. The estimate in \cite{2006-CRAS-BN} was then improved in \cite{2020-APDE-AGMP}, where the discretization is combined with a discrete rearrangement inequality. This results in an energy estimate that is asymptotically optimal and solves completely the cell problem for the case $\gamma=-p$.

On the contrary, in \cite{1998-CPAM-Gobbino,2001-PRSE-GM} the argument relies on a ``horizontal'' discretization \emph{à la Riemann}, namely on considering the restriction of $u$ to the points of the form $x+\delta i$.

Here none of these techniques seem to work, in the sense that the vertical discretization is less effective in simplifying the problem, since the set $\Egpl$ here depends also on the horizontal differences $|y-x|$, while the horizontal discretization with a fixed step produces a trivial estimate in the limit. This is the reason for which we need to introduce the more complicated dyadic decomposition.
%}

\paragraph{\textmd{\textit{Structure of the paper}}}

The rest of this paper is organized as follows. In section~\ref{sec:gatto} we introduce the dyadic functional and we prove its main properties, namely the representation formula of Proposition~\ref{prop:dyadic_repr}, and the estimate from below of Proposition~\ref{prop:est_DF}. In section~\ref{sec:proofs} we prove the estimate from below for the cell problem (Proposition~\ref{prop:gatto}), and we conclude the proof of Theorem~\ref{thm:main}, both in the one-dimensional and in the general case.

%\clearpage

\setcounter{equation}{0}
\section{The dyadic functional}\label{sec:gatto}

This section is the technical core of this paper. Here we introduce a discretized version of (\ref{defn:Fgpl}) and we prove its main properties.

Let $\gamma > 0$, $p\geq 1$, and $\lambda>0$ be real numbers. Let us consider the function
\begin{equation}
\psigpl(\delta,\Delta):=
\begin{cases}
1 &\mbox{if }\Delta > \lambda \delta^{1+\gamma/p}\\
0 &\mbox{if }\Delta \leq \lambda \delta^{1+\gamma/p}
\end{cases}
\qquad\quad
\forall(\delta, \Delta)\in [0,+\infty)^2.
\label{defn:psi-gpl}
\end{equation}

Let $\D:=\{i/2^{k}:i\in \Z, \ k\in\N \}$ denote the set of \emph{dyadic numbers}. For every integer $k\geq 0$, a \emph{$k$-dyadic interval} is any interval of the form
\begin{equation}
\left[\frac{i}{2^k},\frac{i+1}{2^k}\right]
\label{defn:D-int}
\end{equation}
for some $i\in \Z$. For every interval $(a,b)\subseteq\R$ and every integer $k\geq 0$, we consider the $k$-dyadic intervals contained in $(a,b)$, which are indexed by the set
\begin{equation}\label{defn:Dk}
\D_k(a,b):=\left\{i\in \Z: \left[\frac{i}{2^k},\frac{i+1}{2^k}\right]\subseteq(a,b) \right\}.
\end{equation}

For every $\delta>0$ and every function $v:\D\cap(a,b) \to \R$, the \emph{dyadic functional} is defined as
\begin{equation}\label{defn:DF}
\DFgpl(\delta,v,(a,b)):=\sum_{k=0} ^{\infty} \frac{1}{2^{k(\gamma+1)}} \sum_{i\in \D_k(a,b)} \psigpl\left(\frac{\delta}{2^k}, \left|v\left(\frac{i+1}{2^k}\right)-v\left(\frac{i}{2^k}\right)\right| \right).
\end{equation}

We are now ready to state the two properties of the dyadic functional that are the main technical tools in the proof of our main result. The first one is the following representation formula, which shows that in dimension one the functional (\ref{defn:Fgpl}) can be expressed in terms of dyadic functionals.

\begin{prop}[Dyadic representation]\label{prop:dyadic_repr}

Let $\gamma>0$, $p\geq 1$, and $B\geq A$ be real numbers. Let $(a,b)\subseteq\re$ be an interval with length $b-a\geq 1$, and let $u:(a,b)\to [A,B]$ be a measurable function. 

For every $x\in \R$, and every real number $\delta>0$, let us consider the numbers
\begin{equation}
a_{x,\delta}:=\frac{a-x}{\delta}
\qquad\text{and}\qquad
b_{x,\delta}:=\frac{b-x}{\delta},
\label{defn:abxd}
\end{equation}
and the function $v_{x,\delta}:\D\cap (a_{x,\delta},b_{x,\delta})\to \R$ defined by
\begin{equation}\label{defn:v_x,delta}
v_{x,\delta}(q) :=u(x+q\delta) \qquad \forall q\in \D\cap (a_{x,\delta}, b_{x,\delta}).
\end{equation}

Then for every real number $\lambda\geq B-A$ it turns out that
\begin{equation}\label{th:dyadic_repr}
\Fgpl(u,(a,b))=2\lambda^p \int_{1/2} ^1  d\delta \, \delta^{\gamma-1} \int_{0} ^{\delta} \DFgpl\left(\delta,v_{x,\delta},(a_{x,\delta}, b_{x,\delta}) \right)\,dx.
\end{equation}

\end{prop}

The assumption that $b-a\geq 1$ in the previous statement is not really essential, but it simplifies the notation in the proof. The second main tool is an estimate from below for the dyadic functional in terms of the variation of $v$ in $(a,b)$.

\begin{prop}[Asymptotic cost of oscillations for the dyadic functional]\label{prop:est_DF}

Let $\gamma>0$, $p\geq 1$, and $B\geq A$ be real numbers. Let $(a,b)\subseteq\re$ be an interval, and let $v:\D\cap(a,b) \to [A,B]$ be a function.

Let us assume that there exist two integers $\alpha$ and $\beta$, with $a<\alpha<\beta<b$, such that $v(\alpha)=A$ and $v(\beta)=B$.

Then for every $\delta\in[1/2,1]$, and every real number $\lambda\geq 2^{1+\gamma/p}(B-A)$, it turns out that
\begin{equation}\label{th:est_DF}
\DFgpl(\delta,v,(a,b))\geq 
\frac{1}{2^{\gamma+1}-1}\cdot\frac{1}{\lambda^p \delta^{p+\gamma}}
\cdot \frac{(B-A)^p}{(\beta-\alpha)^{p-1}}.
\end{equation}
\end{prop}

%\clearpage

\subsection{Proof of Proposition~\ref{prop:dyadic_repr}}

Let us consider the function defined in (\ref{defn:psi-gpl}). When $N=1$ and $\Omega=(a,b)$ is an interval, we can rewrite (\ref{defn:Fgpl}) in the form
\begin{eqnarray}
\Fgpl(u,(a,b)) & = &
2\lambda^p \int_{a} ^{b} dx \int_{x}^{b} (y-x)^{\gamma-1} \psigpl(y-x,|u(y)-u(x)|)\,dy
\nonumber
\\
& = &
2\lambda^p \int_{0}^{b-a}\delta^{\gamma-1}\,d\delta\int_{a}^{b-\delta} 
\psigpl(\delta,|u(x+\delta)-u(x)|)\,dx.
\label{functional_delta}
\end{eqnarray}

Now by assumption $u$ takes its values in $[A,B]$, and $\lambda\geq B-A$. This means that, for every $\delta\in(1,b-a)$, it turns out that
\begin{equation*}
|u(x+\delta)-u(x)|\leq B-A\leq\lambda\leq \lambda \delta^{1+\gamma/p},
\end{equation*}
for every $x\in(a,b-\delta)$, and hence
\begin{equation*}
\psigpl(\delta,|u(x+\delta)-u(x)|)=0
\end{equation*}
for every $x\in(a,b-\delta)$. As a consequence, in (\ref{functional_delta}) we can restrict ourselves to values $\delta\leq 1$, so that
\begin{equation}
\Fgpl(u,(a,b))=
2\lambda^p \int_{0}^{1}\delta^{\gamma-1}\,d\delta\int_{a}^{b-\delta} 
\psigpl(\delta,|u(x+\delta)-u(x)|)\,dx.
\label{eqn:Fgpl-delta}
\end{equation}

Now from (\ref{defn:abxd}), and (\ref{defn:Dk}) with $k=0$, we obtain that
\begin{equation*}
    \forall i\in\Z
    \qquad
    \left[x\in (a-i\delta,b-(i+1)\delta)
    \ \Longleftrightarrow\ 
    i\in \D_0(a_{x,\delta},b_{x,\delta})\strut\right],
\end{equation*}
and as a consequence
\begin{eqnarray*}
\lefteqn{\hspace{-3em}
\int_{a} ^{b-\delta}  
\psigpl(\delta,|u(x+\delta)-u(x)|)\,dx}
\\
\qquad &=& \sum_{i\in \Z} \int_{(a,b-\delta)\cap (i\delta,(i+1)\delta)} \psigpl(\delta,|u(x+\delta)-u(x)|)\,dx \\
&=& \sum_{i\in \Z} \int_{(0,\delta)\cap (a-i\delta,b-(i+1)\delta)} \psigpl(\delta,|u(x+(i+1)\delta)-u(x+i\delta)|)\,dx\\
&=& \int_{0} ^{\delta} \sum_{i\in \D_0(a_{x,\delta},b_{x,\delta})} \psigpl(\delta,|v_{x,\delta}(i+1)-v_{x,\delta}(i)|)\,dx
\end{eqnarray*}
for every $\delta\in(0,b-a)$. If for every $x\in\R$ and every $\delta>0$ we set for brevity
\begin{equation}
\ggpl(x,\delta):=
\sum_{i\in \D_0(a_{x,\delta},b_{x,\delta})} \psigpl(\delta,|v_{x,\delta}(i+1)-v_{x,\delta}(i)|),
\label{defn:ggpl}
\end{equation}
then (\ref{eqn:Fgpl-delta}) can be written as
\begin{equation*}
\Fgpl(u,(a,b))= 
2\lambda^p \int_{0} ^{1}\delta^{\gamma-1} \,d\delta \int_{0} ^{\delta} \ggpl(x,\delta)\,dx,
\end{equation*}
and in turn this integral can be decomposed as
\begin{eqnarray}
\Fgpl(u,(a,b)) & = & 
2\lambda^p\sum_{k=0}^{\infty}
\int_{2^{-(k+1)}} ^{2^{-k}}\delta^{\gamma-1}\, d\delta \int_{0} ^{\delta} \ggpl(x,\delta)\,dx
\nonumber
\\[1ex]
& = & 
2\lambda^{p}\sum_{k=0}^{\infty}
\int_{1/2}^{1}\frac{\delta^{\gamma-1}}{2^{k\gamma}}\,d\delta \int_{0} ^{\delta/2^{k}} \ggpl\left(x,\frac{\delta}{2^{k}}\right) dx.
\label{eqn:Fgpl-decomp}
\end{eqnarray}

Now we observe that, for every $\delta>0$, the function $x \mapsto \ggpl(x,\delta)$ defined by (\ref{defn:ggpl}) is periodic with period $\delta$. As a consequence, for every non-negative integer  $k$ the function $x \mapsto \ggpl(x,\delta/2^{k})$ is periodic with period $\delta/2^{k}$, and therefore
\begin{equation*}
\int_{0} ^{\delta/2^{k}} \ggpl\left(x,\frac{\delta}{2^{k}}\right)dx
=\frac{1}{2^{k}} \sum_{j=1} ^{2^k} \int_{(j-1)\delta/2^{k}} ^{j\delta/2^{k}} \ggpl\left(x,\frac{\delta}{2^{k}}\right)dx
=\frac{1}{2^{k}} \int_{0} ^{\delta} \ggpl\left(x,\frac{\delta}{2^{k}}\right)dx.
\end{equation*}

Plugging this identity into (\ref{eqn:Fgpl-decomp}) we obtain that
\begin{eqnarray}
\Fgpl(u,(a,b)) & = &    
2\lambda^{p} \sum_{k=0}^{\infty}\int_{1/2}^{1} 
\frac{\delta^{\gamma-1}}{2^{k(\gamma+1)}}\,d\delta \int_{0}^{\delta} \ggpl\left(x,\frac{\delta}{2^{k}}\right)\,dx
\nonumber
\\
& = &
2\lambda^{p} \int_{1/2}^{1}\delta^{\gamma-1}\, d\delta\int_{0}^{\delta}
\left\{\sum_{k=0}^{\infty}\frac{1}{2^{k(\gamma+1)}}\ggpl\left(x,\frac{\delta}{2^{k}}\right)\right\}\,dx.
\label{eqn:Fgpl-sum-ggpl}
\end{eqnarray}

Finally, from (\ref{defn:ggpl}) and (\ref{defn:Dk}) we deduce that

\begin{eqnarray*}
\ggpl\left(x,\frac{\delta}{2^{k}}\right) & = &
\sum_{i\in \D_0\left(a_{x,\delta/2^{k}}, b_{x,\delta/2^{k}}\right)} 
\psigpl\left(\frac{\delta}{2^{k}},|v_{x,\delta/2^{k}}(i+1) -v_{x,\delta/2^{k}}(i)|\right)
\\
& = & 
\sum_{i\in \D_k(a_{x,\delta}, b_{x,\delta})}
\psigpl\left(\frac{\delta}{2^{k}},\left|v_{x,\delta} \left(\frac{i+1}{2^{k}} \right) - v_{x,\delta}\left(\frac{i}{2^{k}}\right) \right|\right),
\end{eqnarray*}
and therefore comparing with (\ref{defn:DF}) we conclude that
\begin{equation*}
%2\lambda^{p}\delta^{\gamma-1}
\sum_{k=0}^{\infty}\frac{1}{2^{k(\gamma+1)}}\ggpl\left(x,\frac{\delta}{2^{k}}\right)=
\DFgpl\left(\delta,v_{x,\delta},(a_{x,\delta}, b_{x,\delta})) \right).
\end{equation*}

Plugging this equality into (\ref{eqn:Fgpl-sum-ggpl}) we obtain exactly (\ref{th:dyadic_repr}).
\qed

%\clearpage

\subsection{Proof of Proposition~\ref{prop:est_DF}}
Let us introduce some notation. We say that two intervals are \emph{essentially disjoint} if they do not share internal points, namely if they intersect at most at the boundary.

We say that a $k$-dyadic interval of the form (\ref{defn:D-int}) is \emph{visible} if
\begin{equation*}
\left|v\left(\frac{i+1}{2^k}\right) - v\left(\frac{i}{2^k}\right) \right| >\lambda \left(\frac{\delta}{2^{k}}\right)^{1+\gamma/p},
\end{equation*}
and we introduce the set
\begin{equation*}
    \Vis_k:=
    \left\{
    i\in\Z:
    \left[\frac{i}{2^k},\frac{i+1}{2^k}\right]\subseteq[\alpha,\beta] 
    \mbox{ is visible} 
    \right\}
\end{equation*}
that indexes all $k$-dyadic intervals that are visible and contained in $[\alpha,\beta]$ (for the sake of shortness, we do not write explicitly the dependence of $\Vis$ on $\gamma$, $p$, $\lambda$, $\delta$, $\alpha$, $\beta$, which are fixed parameters in this proof).

We observe that visible dyadic intervals are exactly those whose endpoints give a positive contribution in the right-hand side of (\ref{defn:DF}), so that
\begin{equation}
\DFgpl(\delta,v,(a,b))\geq
%\DFgpl(\delta,v,(\alpha,\beta))=
\sum_{k=0}^\infty \frac{|\Vis_k|}{2^{k(\gamma+1)}},
\label{est:DF-Vis}
\end{equation}
where as usual the bars denote the number of elements of a set, and the inequality is due to the fact that we are neglecting the contribution of visible intervals that are not contained in $[\alpha,\beta]$.

Now we introduce the set of $k$-dyadic intervals that are \emph{important}. The definition is recursive in $k$. We say that a $k$-dyadic interval of the form (\ref{defn:D-int}) is important if it verifies the following three conditions.
\begin{enumerate}
\renewcommand{\labelenumi}{(\roman{enumi})}
\item It is not visible.

\item At least one of its two halves is visible, namely either
\begin{equation*}
\left|v\left(\frac{2i+1}{2^{k+1}}\right) - v\left(\frac{i}{2^{k}}\right) \right| >\lambda \left(\frac{\delta}{2^{k+1}}\right)^{1+\gamma/p}
\end{equation*}
or
\begin{equation*}
\left|v\left(\frac{i+1}{2^{k}}\right) - v\left(\frac{2i+1}{2^{k+1}}\right) \right| >\lambda \left(\frac{\delta}{2^{k+1}}\right)^{1+\gamma/p}.
\end{equation*}

\item It is not contained in any important $h$-dyadic interval with $h\leq k-1$ (this condition is void when $k=0$).

\end{enumerate}

We also introduce the set
\begin{equation*}
    \Imp_k:=
    \left\{
    i\in\Z:
    \left[\frac{i}{2^k},\frac{i+1}{2^k}\right]\subseteq[\alpha,\beta] 
    \mbox{ is important} 
    \right\}
\end{equation*}
that indexes all $k$-dyadic intervals that are important and contained in $[\alpha,\beta]$.

\paragraph{\textmd{\textit{Properties of visible and important dyadic intervals}}}

In the sequel we need the following properties.
\begin{enumerate}[label={(\arabic*)}]

\item Important dyadic intervals (even with different $k$'s) are pairwise essentially disjoint.

\item  All $0$-dyadic intervals are not visible, namely $\Vis_0=\emptyset$.

\item Every dyadic interval that is visible is contained in a (unique) dyadic interval that is important. More precisely, if $I$ is a $k$-dyadic interval that is visible, then there exists a unique non-negative integer $h\leq k-1$, and a unique $h$-dyadic interval $J$ that is important and such that $I\subseteq J$. 

\item If $I$ is a $k$-dyadic interval that is important, then for every $h\geq k+1$ there exists a $h$-dyadic interval that is visible and contained in $I$.

\end{enumerate}

Let us prove the four properties.

\begin{itemize}

    \item The first one is trivial if the two important intervals are $k$-dyadic intervals with the same $k$. Otherwise, we observe that two dyadic intervals are not essentially disjoint if and only if one is contained in the other, but this is never possible if the two intervals are important because of the third condition in the definition of important dyadic intervals.

    \item  As for the second property, we recall that by assumption $v$ takes its values in $[A,B]$, and moreover $\lambda\geq 2^{1+\gamma/p}(B-A)$ and $\delta\geq 1/2$. If follows that
    \begin{equation*}
    \left|v\left(\frac{i+1}{2^0}\right)-v\left(\frac{i}{2^0}\right)\right|\leq 
    B-A\leq 
    \frac{\lambda}{2^{1+\gamma/p}}\leq
    \lambda\left(\frac{\delta}{2^0}\right)^{1+\gamma/p} \qquad\forall i\in \D_0(a,b),
    \end{equation*}
    which means that the interval $[i,i+1]$ is not visible.

    \item  As for the third property, we argue by induction on $k$. If $k=0$, the conclusion is trivial because there are no visible $0$-dyadic intervals. 
    
    So let us assume that the second property is true up to some $k-1$ and let $I$ be a $k$-dyadic interval of the form (\ref{defn:D-int}) that is visible. Let $I'$ be the unique $(k-1)$-dyadic interval containing $I$, namely either $I'=[i/2^{k},(i+2)/2^{k}]$ if $i$ is even, or $I'=[(i-1)/2^{k},(i+1)/2^{k}]$ if $i$ is odd. Then there are two cases.
    \begin{itemize}
        \item If $I'$ is visible, by the inductive assumption it is contained in some important $h$-dyadic interval, and this important interval also contains $I$.

        \item If $I'$ is not visible, then it satisfies the first two properties in the definition of important intervals. Hence, either it is important (so it is an important interval containing $I$) or it is contained in some important $h$-dyadic interval with $h\leq k-2$ (and this important interval also contains $I$).
    \end{itemize}

    This concludes the proof of the inductive step, hence the existence part of the third property. Uniqueness follows from the first statement.

    \item As for the last property, we argue by induction on $h\geq k+1$. The case $h=k+1$ is an immediate consequence of the second condition in the definition of important intervals.

    So let us assume that $I$ contains a visible $h$-dyadic interval $J=[j/2^h, (j+1)/2^h]$ for some $h\geq k+1$. Then we claim that at least one of the two halves of $J$ is visible, and hence $I$ contains a visible $(h+1)$-dyadic interval. Indeed, if this were false, then we would deduce that (here we exploit that $\gamma>0$)
    \begin{multline*}
    \left|v\left(\frac{j+1}{2^h}\right)-v\left(\frac{j}{2^h}\right)\right|
    \leq
    \left|v\left(\frac{j+1}{2^h}\right)-v\left(\frac{2j+1}{2^{h+1}}\right)\right| + 
    \left|v\left(\frac{2j+1}{2^{h+1}}\right)-v\left(\frac{j}{2^h}\right)\right|
    \\[0.5ex]
    \leq 2\lambda \left(\frac{\delta}{2^{h+1}}\right)^{1+\gamma/p}
    \leq\lambda \left(\frac{\delta}{2^h}\right)^{1+\gamma/p}.
    \qquad\quad
    \end{multline*}

    But this means that $J$ is not visible, which is a contradiction.

\end{itemize}

\paragraph{\textmd{\textit{Estimate from above for $B-A$}}}

We show that
\begin{equation}\label{est_J}
(B-A)^p\leq \lambda^p\delta^{p+\gamma}\cdot (\beta-\alpha)^{p-1}\cdot \sum_{k=0}^{\infty}\frac{|\Imp_k|}{2^{k(\gamma+1)}}.
\end{equation}

To this end, let us fix a positive integer $m$, and let us consider all important $k$-dyadic intervals with $k\leq m$ that are contained in $[\alpha,\beta]$. The complement of their union in $[\alpha,\beta]$ is the union of a suitable set of $m$-dyadic intervals. More precisely, we obtain a partition of the form
\begin{equation*}
[\alpha,\beta]= 
\bigcup_{k=0} ^{m} \bigcup_{i\in \Imp_k} \left[\frac{i}{2^k},\frac{i+1}{2^k}\right]
\cup \bigcup_{i\in R_m} \left[\frac{i}{2^m},\frac{i+1}{2^m}\right],
\end{equation*}
where $R_m\subseteq\D_m(a,b)$ is a suitable index set, and all the intervals that appear in the right-hand side are pairwise essentially disjoint.

The key point is that all the intervals in the right-hand side are not visible. Indeed, important dyadic intervals are not visible because of the first condition in their definition, while all $m$-dyadic intervals that are visible are contained in an important $k$-dyadic interval for some $k\leq m-1$, thanks to the property~(3) of the previous paragraph, and therefore they do not appear in the union indexed by $R_m$.

Since all these intervals are not visible, we can estimate the difference between the values on $v$ at the endpoints of each interval in terms of the length of the interval itself. We obtain that
\begin{eqnarray*}
B-A&=&|v(\beta)-v(\alpha)|\\
&\leq & \sum_{k=0} ^{m}\sum_{i\in \Imp_k} \left|v\left(\frac{i+1}{2^k}\right) - v\left(\frac{i}{2^k}\right)\right| + \sum_{i\in R_m} \left|v\left(\frac{i+1}{2^m}\right) - v\left(\frac{i}{2^m}\right)\right|\\
&\leq & \sum_{k=0} ^{m}\sum_{i\in \Imp_k} \lambda\left(\frac{\delta}{2^k}\right)^{1+\gamma/p} + \sum_{i\in R_m} \lambda\left(\frac{\delta}{2^m}\right)^{1+\gamma/p}\\
&= & \sum_{k=0} ^{m} \lambda |\Imp_k| \left(\frac{\delta}{2^k}\right)^{1+\gamma/p} + |R_m| \lambda\left(\frac{\delta}{2^m}\right)^{1+\gamma/p}.
\end{eqnarray*}

Now we observe that $|R_m|\leq (\beta-\alpha)\cdot 2^m$, and therefore letting $m\to +\infty$ we deduce that (here again we exploit that $\gamma>0$)
\begin{equation*}
B-A\leq \lambda \delta^{1+\gamma/p} \sum_{k=0}^{\infty} \frac{|\Imp_k|}{2^{k(1+\gamma/p)}}.
\end{equation*}

At this point by Hölder inequality we conclude that
\begin{eqnarray*}
B-A &\leq & \lambda \delta^{1+\gamma/p} \left(\sum_{k=0}^{\infty} \frac{|\Imp_k|}{2^{k(1+\gamma)}}\right)^{1/p} \left(\sum_{k=0}^{\infty} \frac{|\Imp_k|}{2^k}\right)^{(p-1)/p}\\
&\leq & \lambda \delta^{1+\gamma/p} \left(\sum_{k=0}^{\infty} \frac{|\Imp_k|}{2^{k(1+\gamma)}}\right)^{1/p} (\beta-\alpha)^{(p-1)/p},
\end{eqnarray*}
which is equivalent to (\ref{est_J}).

\paragraph{\textmd{\textit{Conclusion}}}

From (\ref{est:DF-Vis}) and property~(2) of visible dyadic intervals we know that
\begin{equation}\label{DF=sum_vis}
\DFgpl(\delta,v,(a,b))\geq\sum_{k=1}^\infty \frac{1}{2^{k(\gamma+1)}} |\Vis_k|.
\end{equation}

Moreover, from properties~(1) and~(4) of dyadic intervals we deduce that
\begin{equation*}
|\Vis_k| \geq \sum_{h=0} ^{k-1} |\Imp_h|,
\end{equation*}

Plugging this inequality into (\ref{DF=sum_vis}), and rearranging the addenda, we obtain that
% \begin{eqnarray*}
%     \DFgpl(\delta,v,(a,b)) & \geq &
%     \sum_{k=1}^\infty\frac{1}{2^{k(\gamma+1)}}\sum_{h=0}^{k-1} |\Imp_h|
%     \\
%     & = &  
%     \sum_{h=0}^\infty |\Imp_h| \sum_{k=h+1}^\infty \frac{1}{2^{k(\gamma+1)}}
%     \\
%     & = &
%     \frac{1}{2^{\gamma+1}-1} \sum_{h=0}^\infty
%     \frac{|\Imp_h|}{2^{h(\gamma+1)}}.
% \end{eqnarray*}
% [oppure]
\begin{multline*}
    \qquad
    \DFgpl(\delta,v,(a,b)) \geq 
    \sum_{k=1}^\infty\frac{1}{2^{k(\gamma+1)}}\sum_{h=0}^{k-1} |\Imp_h|
    \\
    =\sum_{h=0}^\infty |\Imp_h| \sum_{k=h+1}^\infty \frac{1}{2^{k(\gamma+1)}}=
    \frac{1}{2^{\gamma+1}-1} \cdot\sum_{h=0}^\infty
    \frac{|\Imp_h|}{2^{h(\gamma+1)}}.
    \qquad
\end{multline*}

Recalling (\ref{est_J}), this inequality implies (\ref{th:est_DF}).
\qed

%\clearpage

\setcounter{equation}{0}
\section{Proof of the main result}\label{sec:proofs}

In this section we prove Theorem~\ref{thm:main}. Following the classical approach, we consider the cell problem that consists in computing, in dimension one, the minimal energy of a function that oscillates between two values $A$ and $B$ in some interval $(a,b)$. In the following result we estimate this minimal energy by relying on the dyadic theory of section~\ref{sec:gatto}. Throughout this section $\Leb$ denotes the one-dimensional Lebesgue measure.

%\clearpage

\begin{prop}[Lower bound for the cell problem]\label{prop:gatto}

Let $\gamma>0$, $p\geq 1$, and $B\geq A$ be real numbers. Let $k$ be a non-negative integer, let $(a,b)\subseteq\re$ be an interval with length $b-a\geq 3\cdot 2^{-k}$, and let $u:(a,b)\to [A,B]$ be a measurable function. 

Let us assume that there exists a real number $\ep>0$ such that
\begin{equation}\label{hp:density_ep}
\Leb\left(\{ x\in (a,a+2^{-k}): u(x)\neq A \}\cup 
\{ x\in (b-2^{-k},b): u(x)\neq B \}\right)\leq 
\frac{\ep}{2}\cdot 2^{-k}.
\end{equation}

Then for every real number $\lambda \geq 2^{(k+1)(1+\gamma/p)}(B-A)$ it holds that
\begin{equation}\label{th:gatto}
\Fgpl (u,(a,b))\geq (1-\ep) 2c_\gamma \cdot\frac{(B-A)^p}{(b-a)^{p-1}},
\end{equation}
where $c_\gamma$ is the constant defined by (\ref{defn:c_gamma}).
\end{prop}

\begin{proof}

To begin with, we prove the result when $k=0$, and then we deduce the general case through a homothety argument.

\subparagraph{\textmd{\textit{Case $k=0$}}}

Let us fix $\delta\in[1/2,1]$ and $x\in(a,b)$. Let us consider the interval $(a_{x,\delta},b_{x,\delta})$ whose endpoints are defined by (\ref{defn:abxd}), and let us denote by $[\alpha_{x,\delta},\beta_{x,\delta}]$ the largest closed interval with integer endpoints contained in $(a_{x,\delta},b_{x,\delta})$, namely the interval with endpoints
\begin{equation*}
    \alpha_{x,\delta}:=\lfloor a_{x,\delta}\rfloor+1
    \qquad\text{and}\qquad
    \beta_{x,\delta}:=\lceil b_{x,\delta}\rceil-1.
\end{equation*}
(note that $\beta_{x,\delta}\geq\alpha_{x,\delta}+1$ because when $k=0$ we assumed that $b-a\geq 3$).

Now let $X_\delta$ denote the sets of points $x\in(0,\delta)$ such that
\begin{equation*}
v_{x,\delta}(\alpha_{x,\delta})=u(x+\alpha_{x,\delta}\delta)=A
\qquad \text{and} \qquad
v_{x,\delta}(\beta_{x,\delta})=u(x+\beta_{x,\delta}\delta)=B.
\end{equation*}

We observe that
\begin{equation*}
    x+\alpha_{x,\delta}\delta=
    \begin{cases}
    x+\delta(\lfloor a/\delta \rfloor+1) &
    \mbox{if }x\in \left(0,\delta\{a/\delta\}\right),
    \\[0.5ex]
    x+\delta \lfloor a/\delta \rfloor &
    \mbox{if }x\in \left(\delta\{a/\delta\},\delta\right),
    \end{cases}
    %=x+\delta \left(\left\lfloor\frac{a-x}{\delta} \right\rfloor+1\right)
    %\qquad
    %\forall x\in(0,\delta),
\end{equation*}
where $\{a/\delta\}$ denotes the fractional part of $a/\delta$. This shows that $x\mapsto x+\alpha_{x,\delta}\delta$ is a measure preserving map from $(0,\delta)$ to $(a,a+\delta)$, and hence
\begin{equation*}
    \Leb\left(\{x\in(0,\delta):v_{x,\delta}(\alpha_{x,\delta})\neq A\}\right)=
    \Leb\left(\{y\in(a,a+\delta):u(y)\neq A\}\right).
\end{equation*}

In an analogous way we obtain that
\begin{equation*}
    \Leb\left(\{x\in(0,\delta):v_{x,\delta}(\beta_{x,\delta})\neq B\}\right)=
    \Leb\left(\{y\in(b-\delta,b):u(y)\neq B\}\right).
\end{equation*}

Since
\begin{equation*}
    (0,\delta)\setminus X_{\delta}=
    \left\{x\in(0,\delta):v_{x,\delta}(\alpha_{x,\delta})\neq A\right\}\cup
    \left\{x\in(0,\delta):v_{x,\delta}(\beta_{x,\delta})\neq B\right\},
\end{equation*}
from assumption (\ref{hp:density_ep}) with $k=0$ we deduce that
\begin{eqnarray*}
    \Leb((0,\delta)\setminus X_{\delta}) & \leq &
    \Leb\left(\{y\in(a,a+\delta):u(y)\neq A\}\right)+
    \Leb\left(\{y\in(b-\delta,b):u(y)\neq B\}\right)
    \\
    & \leq &
    \Leb\left(\{y\in(a,a+1):u(y)\neq A\}\right)+
    \Leb\left(\{y\in(b-1,b):u(y)\neq B\}\right)
    \\
    & \leq &
    \frac{\ep}{2},
\end{eqnarray*}
and hence, recalling that $\delta\geq 1/2$,
\begin{equation}\label{leb_Xdelta}
\Leb(X_\delta)\geq \delta - \frac{\ep}{2} \geq (1-\ep)\delta.
\end{equation}

Now we observe that, for every $x$ in $X_\delta$, the function $v_{x,\delta}$ defined by (\ref{defn:v_x,delta}) satisfies the assumptions of Proposition~\ref{prop:est_DF} with $(a,b):=(a_{x,\delta},b_{x,\delta})$ and $[\alpha,\beta]:=[\alpha_{x,\delta},\beta_{x,\delta}]$, so that from (\ref{th:est_DF}) we obtain that
\begin{equation}
\DFgpl(\delta,v_{x,\delta},(a_{x,\delta},b_{x,\delta}))\geq 
\frac{1}{2^{\gamma+1}-1}\cdot\frac{1}{\lambda^p \delta^{p+\gamma}}
\cdot\frac{(B-A)^p}{(\beta_{x,\delta}-\alpha_{x,\delta})^{p-1}}
\label{est:DF-abxd}
\end{equation}
whenever $\lambda\geq 2^{1+\gamma/p}(B-A)$, which is exactly our assumption on $\lambda$ in the case $k=0$. Since
\begin{equation*}
\beta_{x,\delta}-\alpha_{x,\delta}\leq 
b_{x,\delta}-a_{x,\delta} = 
\frac{b-a}{\delta},
\end{equation*}
from (\ref{est:DF-abxd}) we conclude that 
\begin{equation*}
\DFgpl(\delta,v_{x,\delta},(a_{x,\delta},b_{x,\delta}))\geq 
\frac{1}{2^{\gamma+1}-1}\cdot\frac{1}{\lambda^p \delta^{1+\gamma}}
\cdot\frac{(B-A)^p}{(b-a)^{p-1}}
\end{equation*}
for every $\delta\in[1/2,1]$, every $x\in X_\delta$, and every $\lambda\geq 2^{1+\gamma/p}(B-A)$. At this point from Proposition~\ref{prop:dyadic_repr} we deduce that
\begin{eqnarray*}
\Fgpl(u,(a,b)) & = & 
2\lambda^p \int_{1/2} ^1  d\delta \, \delta^{\gamma-1} \int_{0} ^{\delta} \DFgpl\left(\delta,v_{x,\delta},(a_{x,\delta}, b_{x,\delta}) \right)\,dx
\\
& \geq & 
2\lambda^p \int_{1/2} ^1  d\delta \, \delta^{\gamma-1} \int_{X_\delta} \DFgpl\left(\delta,v_{x,\delta},(a_{x,\delta}, b_{x,\delta}) \right)\,dx
\\
& \geq & 
\frac{2}{2^{\gamma+1}-1}\cdot\frac{(B-A)^p}{(b-a)^{p-1}} 
\int_{1/2}^{1}\frac{\Leb(X_\delta)}{\delta^2}\,d\delta.
\end{eqnarray*}

Recalling (\ref{leb_Xdelta}) we conclude that
\begin{equation*}
\Fgpl(u,(a,b))\geq 
(1-\ep)\frac{2}{2^{\gamma+1}-1}\cdot\frac{(B-A)^p}{(b-a)^{p-1}} \int_{1/2} ^1 \frac{1}{\delta}\, d\delta=
(1-\ep) \frac{2\log 2}{2^{\gamma+1}-1}\cdot\frac{(B-A)^p}{(b-a)^{p-1}},
\end{equation*}
which is exactly (\ref{th:gatto}).

\subparagraph{\textmd{\textit{Case $k\geq 1$}}}

Let us consider the function $u_k(x):=u(2^{-k}x)$. With a change of variables we obtain that
\begin{equation}\label{eq:change_of_variables_2^k}
\Fgpl(u,(a,b))=2^{k(p-1)}F_{\gamma,p,\lambda_k}(u_k,(a_k,b_k)),
\end{equation}
where %$\lambda_k=2^{-k(1+\gamma/p)}\lambda$, $a_k=2^k a$ and $b_k=2^k b$.
\begin{equation*}
    \lambda_k:=2^{-k(1+\gamma/p)}\lambda,
    \qquad\quad
    a_k:=2^k a,
    \qquad\quad
    b_k:=2^k b.
\end{equation*}

We observe that the assumptions on $b-a$ and $\lambda$ imply in this case that $b_k-a_k\geq 3$ and $\lambda_k\geq 2^{1+\gamma/p}(B-A)$, while (\ref{hp:density_ep}) implies that
\begin{equation*}
\Leb\left(\{ x\in (a_k,a_k+1): u_k(x)\neq A \}\cup 
\{ x\in (b_k-1,b_k): u_k(x)\neq B \}\strut\right)\leq 
\frac{\ep}{2}.
\end{equation*}

Hence, from the case $k=0$, we conclude that
\begin{equation*}
F_{\gamma,p,\lambda_k}(u_k,(a_k,b_k))\geq 
(1-\ep) 2 c_\gamma \cdot\frac{(B-A)^p}{(b_k-a_k)^{p-1}}=
(1-\ep)2 c_\gamma\cdot\frac{1}{2^{k(p-1)}}\cdot\frac{(B-A)^p}{(b-a)^{p-1}}.
\end{equation*}

Plugging this inequality into (\ref{eq:change_of_variables_2^k}) we obtain (\ref{th:gatto}) also in this case.
\end{proof}

%\clearpage

\subsection{Proof of Theorem~\ref{thm:main} in dimension one}\label{sec:proof-1D}

In this section we prove Theorem~\ref{thm:main} in the case $N=1$. To this end, we consider any open set $\Omega\subseteq \R$, any function $u \in L^1 _{\loc}(\Omega)$, and any family $\{u_{\lambda}\} \subseteq L^1 _{\loc} (\Omega)$ such that $u_{\lambda}\to u$ in $L^1 _{\loc} (\Omega)$ as $\lambda\to +\infty$. Since $C_{1,p}=2$ for every $p\geq 1$, in this case we need to prove that
\begin{equation}%\label{th:liminf_inequality}
\liminf_{\lambda\to +\infty} \Fgpl(u_\lambda,\Omega)\geq 2 c_{\gamma} F_p(u,\Omega),
\label{th:1D}
\end{equation}
where $c_\gamma$ is the constant defined by (\ref{defn:c_gamma}). 

\paragraph{\textmd{\textit{Measure estimate near Lebesgue points}}}

Let $[c,d]\subseteq\Omega$ be an interval whose endpoints $c$ and $d$ are Lebesgue points of $u$. We claim that, for every $\ep>0$, there exists an integer $k_{\ep}\geq 0$ and a real number $\lambda_\ep>0$ such that $d-c\geq 3\cdot 2^{-k_\ep}$, and in addition
\begin{equation}
    \Leb\left(\{x\in(c,c+2^{-k_\ep}):u_{\lambda}(x)>u(c)+2\ep\}\right)<
    \frac{\ep}{4}\cdot 2^{-k_\ep}
    \qquad
    \forall\lambda\geq\lambda_\ep,
    \label{th:leb-c}
\end{equation}
and
\begin{equation}
    \Leb\left(\{x\in(d-2^{-k_\ep},d):u_{\lambda}(x)<u(d)-2\ep\}\right)<
    \frac{\ep}{4}\cdot 2^{-k_\ep}
    \qquad
    \forall\lambda\geq\lambda_\ep,
    \label{th:leb-d}
\end{equation}

To this end, we begin by observing that the definition of Lebesgue point implies that
\begin{equation*}
    \lim_{k\to +\infty}
    \frac{\Leb\left(\{x\in(c,c+2^{-k}):u(x)>u(c)+\ep\}\right)}{2^{-k}}=
    0,
\end{equation*}
and hence
\begin{equation}
    \Leb\left(\{x\in(c,c+2^{-k}):u(x)>u(c)+\ep\}\right)\leq
    \frac{\ep}{8}\cdot 2^{-k}
    \label{defn:leb-pt-c}
\end{equation}
for every $k$ large enough. In the same way we obtain that
\begin{equation}
    \Leb\left(\{x\in(d-2^{-k},d):u(x)<u(d)-\ep\}\right)\leq
    \frac{\ep}{8}\cdot 2^{-k}
    \label{defn:leb-pt-d}
\end{equation}
when $k$ is large enough. Now we choose an integer $k_\ep$, with $d-c\geq 3\cdot 2^{-k_\ep}$, such that both (\ref{defn:leb-pt-c}) and (\ref{defn:leb-pt-d}) hold true, and we claim that we can find $\lambda_\ep$ such that both (\ref{th:leb-c}) and (\ref{th:leb-d}) hold true. 

Indeed, let us assume by contradiction that there exists a sequence $\lambda_i\to +\infty$ such that one of them, say (\ref{th:leb-c}) without loss of generality, is false, namely
\begin{equation}
    \Leb\left(\{x\in(c,c+2^{-k_\ep}):u_{\lambda_i}(x)>u(c)+2\ep\}\right)\geq
    \frac{\ep}{4}\cdot 2^{-k_\ep}
    \label{th:leb-absurd}
\end{equation}
for every index $i$. Now we observe that
\begin{eqnarray*}
    \lefteqn{
    \left\{x\in(c,c+2^{-k_\ep}):u_{\lambda_i}(x)-u(x)\geq\ep\right\}
    }
    \\[0.5ex]
    & \subseteq & 
    \left\{x\in(c,c+2^{-k_\ep}):u_{\lambda_i}(x)\geq u(c)+2\ep,\ u(x)\leq u(c)+\ep\right\}
    \\[0.5ex]
    & = &
    \left\{x\in(c,c+2^{-k_\ep}):u_{\lambda_i}(x)\geq u(c)+2\ep\right\}
    \setminus
    \left\{x\in(c,c+2^{-k_\ep}):u(x)> u(c)+\ep\right\},
\end{eqnarray*}
and therefore from (\ref{th:leb-absurd}) and (\ref{defn:leb-pt-c}) it follows that
\begin{equation*}
    \Leb\left(\left\{x\in(c,c+2^{-k_\ep}):u_{\lambda_i}(x)-u(x)\geq\ep\right\}\right)\geq
    \frac{\ep}{4}\cdot 2^{-k_\ep}-\frac{\ep}{8}\cdot 2^{-k_\ep}=
    \frac{\ep}{8}\cdot 2^{-k_\ep},
\end{equation*}
which in turn implies that
\begin{equation*}
    \int_{c}^{c+2^{-k_\ep}}|u_{\lambda_i}(x)-u(x)|\,dx\geq
    \frac{\ep}{8}\cdot 2^{-k_\ep}\cdot 2^{-k_\ep}
\end{equation*}
for every index $i$, which contradicts the fact that $u_{\lambda_i}\to u$ in $L^1_{\loc}(\Omega)$.

\paragraph{\textmd{\textit{Estimate from below between Lebesgue points}}}

Let us assume that $\Omega=(a,b)$ is an interval, and let $c<d$ be two Lebesgue points of $u$ in $(a,b)$. We claim that
\begin{equation}
\liminf_{\lambda \to +\infty} \Fgpl(u_\lambda,(c,d))\geq 2c_\gamma \frac{|u(d)-u(c)|^p}{(d-c)^{p-1}}.
\label{liminf_leb_pts}
\end{equation}

To begin with, we observe that (\ref{liminf_leb_pts}) is trivial if $u(d)=u(c)$, and otherwise we can assume, up to a change of sign, that $u(d)> u(c)$. 

Now we fix a positive real number $\ep< (u(d)-u(c))/4$, we set $A_\ep:=u(c)+2\ep$ and $B_\ep:=u(d)-2\ep$, and we consider the truncated function
\begin{equation*}
    u_{\ep,\lambda}(x):=\min\{\max\{u_{\lambda}(x),A_\ep\},B_\ep\}.
\end{equation*}

We observe that
\begin{equation*}
    |u_{\ep,\lambda}(y)-u_{\ep,\lambda}(x)|\leq |u_{\lambda}(y)-u_{\lambda}(x)|
    \qquad
    \forall (x,y)\in (c,d)^2,
\end{equation*}
and therefore
\begin{equation}
\Fgpl(u_{\lambda},(c,d))\geq \Fgpl(u_{\ep,\lambda},(c,d)). 
\label{th:Fu>Fuep}
\end{equation}

The key point is that, thanks to (\ref{th:leb-c}) and  (\ref{th:leb-d}), the function $u_{\ep,\lambda}$ satisfies the assumptions of Proposition~\ref{prop:gatto} in the interval $(c,d)$, with parameters
\begin{equation*}
    A:=A_\ep,
    \qquad
    B:=B_\ep,
    \qquad
    k:=k_\ep,
\end{equation*}
and therefore from (\ref{th:gatto}) we deduce that
\begin{equation*}
\Fgpl(u_{\ep,\lambda},(c,d))\geq 
(1-\ep) 2 c_\gamma \cdot\frac{(B_\ep-A_\ep)^p}{(b-a)^{p-1}}=
(1-\ep) 2c_\gamma \cdot\frac{(u(d)-u(c)-4\ep)^p}{(b-a)^{p-1}},
\end{equation*}
for every
\begin{equation*}
\lambda\geq\max \left\{\lambda_\ep , 2^{(k_\ep+1)(1+\gamma/p)}(B_\ep-A_\ep)\right\}.
\end{equation*}

Keeping (\ref{th:Fu>Fuep}) into account, this implies that
\begin{equation*}
\liminf_{\lambda\to +\infty} \Fgpl(u_{\lambda},(c,d))\geq 
(1-\ep) 2c_\gamma \cdot\frac{(u(d)-u(c)-4\ep)^p}{(b-a)^{p-1}}.
\end{equation*}

Letting $\ep\to 0^+$ we obtain (\ref{liminf_leb_pts}).

\paragraph{\textmd{\textit{Estimate from below in an interval}}}

Let us assume again that $\Omega=(a,b)$ is an interval, and let $x_1<\ldots < x_m$ be Lebesgue points of $u$ in $(a,b)$. Applying (\ref{liminf_leb_pts}) in each of the intervals $(x_{i},x_{i+1})$ we deduce that
\begin{equation*}
\liminf_{\lambda\to +\infty} \Fgpl(u_{\lambda},(a,b))\geq
\liminf_{\lambda\to +\infty} \sum_{i=1} ^{m-1} \Fgpl(u_{\lambda},(x_{i},x_{i+1})) \geq
2c_\gamma \sum_{i=1} ^{m-1} \frac{(u(x_{i+1})-u(x_{i}))^p}{(x_{i+1}-x_{i})^{p-1}},
\end{equation*}
and we conclude by recalling that the right-hand side of (\ref{th:1D}) is the supremum of the right-hand side of the latter over all possible choices of the Lebesgue points $x_1$, \ldots, $x_m$ (see for example \cite[Lemma~3.5]{2020-APDE-AGMP}).

\paragraph{\textmd{\textit{Estimate from below in a general open set}}}

If $\Omega\subseteq \R$ is any bounded open set, we can write it as the union of pairwise disjoint open intervals (its connected components), namely 
\begin{equation*}
\Omega=\bigcup_{i\in I} (a_i,b_i),
\end{equation*}
for some finite or countable set of indices $I$. Applying (\ref{th:1D}) in each of these intervals we obtain that
\begin{multline*}
\qquad
\liminf_{\lambda\to +\infty} \Fgpl(u_\lambda,\Omega)\geq
\sum_{i\in I} \liminf_{\lambda\to +\infty} \Fgpl(u_\lambda,(a_i,b_i))
\\
\geq\sum_{i\in I} 2c_\gamma F_p(u,(a_i,b_i))=
2c_\gamma F_p(u,\Omega),
\qquad
\end{multline*}
which completes the proof of (\ref{th:1D}) when $\Omega$ is bounded. If $\Omega$ is unbounded, we apply the result in $\Omega\cap(-R,R)$ and then we let $R\to +\infty$. 
\qed

\subsection{Proof of Theorem~\ref{thm:main} in higher dimension}\label{sec:proof-HD}

In this section we prove Theorem~\ref{thm:main} in any space dimension. The argument relies on a standard slicing technique, and it is the same used in \cite{2023-Picenni,2023-Lahti}.

Let $\Omega \subseteq \R^N$ be an open set, let $\sigma\in \s^{N-1}$ be a unit vector, and let $\sigma^\perp$ be the subspace of $\R^N$ orthogonal to $\sigma$. Let $\Omega_\sigma \subseteq \sigma^\perp$ denote the orthogonal projection of $\Omega$ onto $\sigma^\perp$.

For every $y\in \Omega_\sigma$ we set $\Omega_{y,\sigma}:=\{t\in\R: y+\sigma t \in \Omega \}$, and we consider the one-dimensional slice $u_{y,\sigma}:\Omega_{y,\sigma} \to \R$  of $u$ defined by $u_{y,\sigma}(t):=u(y+\sigma t)$.

The theory in \cite{AFP} implies that
\begin{equation}\label{sectioning_norm}
\int_{\s^{N-1}} d\mathcal{H}^{N-1}(\sigma) 
\int_{\Omega_\sigma} F_p(u_{y,\sigma}, \Omega_{y,\sigma}) \,dy = 
C_{N,p}\cdot F_p (u,\Omega)
\end{equation}
for every $p\geq 1$ and every $u\in L^1_\loc(\Omega)$. Moreover, with a change of variables (see the general formula in \cite[page~622]{2020-APDE-AGMP}) we obtain that
\begin{equation}\label{sectioning_functionals}
\Fgpl(u,\Omega)=
\frac{1}{2}\int_{\s^{N-1}} d\mathcal{H}^{N-1}(\sigma) 
\int_{\Omega_\sigma} \Fgpl(u_{y,\sigma},\Omega_{y,\sigma})\,dy,
\end{equation}
again for every $u\in L^1_\loc(\Omega)$ and every admissible value of $\gamma$, $p$, $\lambda$ (more details can be found in \cite[Lemma~3.1]{2023-Picenni} when $p=1$ and $\Omega=\R^N$, but the same argument works also in this more general case).

Now let $\{u_\lambda\}\subseteq L^1_{\loc} (\Omega)$ be a family such that $u_{\lambda}\to u$ in $L^1_{\loc} (\Omega)$ and
$$\liminf_{\lambda\to +\infty} \Fgpl(u_{\lambda},\Omega)<+\infty.$$

Then there exists a sequence $\lambda_n\to +\infty$ such that 
$$\liminf_{\lambda\to +\infty} \Fgpl(u_{\lambda},\Omega) = 
\lim_{n\to +\infty} F_{\gamma,p,\lambda_n}(u_{\lambda_n},\Omega),$$
and such that the sequence $u_{\lambda_n}(x)$ converges to $u(x)$ for almost every $x\in\Omega$, and is dominated by some function in $L^1_\loc(\Omega)$. In particular, when we consider this sequence, for every $\sigma\in \s^{N-1}$ it holds that $(u_{\lambda_n})_{y,\sigma}\to u_{y,\sigma}$ in $L^1_{\loc} (\Omega_{y,\sigma})$ for almost every $y\in \Omega_\sigma$.

Therefore, from the integral geometric representations (\ref{sectioning_norm}) and (\ref{sectioning_functionals}), Fatou's lemma, and the one-dimensional result, we conclude that
\begin{eqnarray*}
\liminf_{\lambda\to +\infty} \Fgpl(u_{\lambda},\Omega)
&=&\lim_{n\to +\infty} F_{\gamma,p,\lambda_n}(u_{\lambda_n},\Omega)
\\
&=&\lim_{n\to +\infty} \frac{1}{2}\int_{\s^{N-1}} d\mathcal{H}^{N-1}(\sigma) \int_{\Omega_\sigma} F_{\gamma,p,\lambda_n}((u_{\lambda_n})_{y,\sigma},\Omega_{y,\sigma})\,dy\\
&\geq & \frac{1}{2}\int_{\s^{N-1}} d\mathcal{H}^{N-1}(\sigma) \int_{\Omega_\sigma} \left\{ \liminf_{n\to +\infty} F_{\gamma,p,\lambda_n}((u_{\lambda_n})_{y,\sigma},\Omega_{y,\sigma})\right\}dy\\
&\geq & \frac{1}{2}\int_{\s^{N-1}} d\mathcal{H}^{N-1}(\sigma) \int_{\Omega_\sigma} 2c_\gamma F_p(u_{y,\sigma}, \Omega_{y,\sigma}) \,dy\\
&=& C_{N,p}\cdot c_\gamma\cdot F_p(u,\Omega),
\end{eqnarray*}
which completes the proof.
\qed

\subsubsection*{\centering Acknowledgments}

The authors are grateful to H.~Brezis for sharing with them preliminary versions of~\cite{Brezis-Lincei,BSVY}, and for encouraging them to investigate the open problems stated therein. 

The authors are members of the \selectlanguage{italian} ``Gruppo Nazionale per l'Analisi Matematica, la Probabilità e le loro Applicazioni'' (GNAMPA) of the ``Istituto Nazionale di Alta Matematica'' (INdAM).

The authors acknowledge the MIUR Excellence Department Project awarded to the Department of Mathematics, University of Pisa, CUP I57G22000700001.

\selectlanguage{english}

%\bibliographystyle{MaxNew}
%\bibliography{Bibliography}

% Bibliografia by MaxNew

\label{NumeroPagine}

\end{document}